\documentclass[reqno]{amsart}
\usepackage{amsthm}
\usepackage{amsfonts,amscd, amssymb}
\usepackage{hyperref}
\usepackage{xcolor}
\usepackage{tikz-cd}
\usepackage{amsthm,latexsym,euscript,amscd}

\usetikzlibrary{matrix}
\newtheorem{theorem}{Theorem}[section]
\newtheorem{lemma}[theorem]{Lemma}
\newtheorem{corollary}[theorem]{Corollary}
\newtheorem{proposition}[theorem]{Proposition}
\newtheorem{remark}[theorem]{Remark}
\theoremstyle{definition}

\numberwithin{equation}{section} 
\makeatother

\DeclareMathOperator{\Idb}{{\mathbb I}}

\DeclareMathOperator{\Rdb}{{\mathbb R}}

\DeclareMathOperator{\Al}{{\mathcal A}}

\numberwithin{equation}{section}

\begin{document}

\title[Operator inequalities for relative operator entropies]{Refined operator inequalities for relative operator entropies}

\author{Shuzhou Wang}
\address{Department of Mathematics, University of Georgia, Athens, GA, 30602}
\email{szwang@uga.edu}
\author{Zhenhua Wang}
\address{Department of Mathematics, University of Georgia, Athens, GA, 30602}
\email{ezrawang@uga.edu}

\date{}

\begin{abstract}
 In this paper, we investigate the relative operator entropies in the more general settings of C*-algebras, real C*-algebras and JC-algebras. We show that all the operator inequalities on relative operator entropies still hold in these broader settings. In addition, we improve the lower and upper bounds of the relative operator $(\alpha, \beta)$-entropy established by Nikoufar in \cite{NIKOUFAR2014376} and \cite{Nikoufar2020} 
which refined the bounds for the relative operator entropy obtained by 
Fujii and Kamei \cite{fujii1989relative, fujii1989uhlmann}.
\end{abstract}

\maketitle

\section{Introduction}
It is well known that entropy first appeared in thermodynamics. It plays an important role in quantum mechanics and is a fundamental notion for quantum information theory. In \cite{Segal1960}, Segal gave a mathematical formulation of the entropy of a state of a semi-finite von Neumann algebra, which included the cases of the information theory and quantum statistical mechanics. To investigate the background of Segal's definition, Nakamura and Umegaki \cite{nakamura1961} considered the operator entropy $-A\log(A)$ for positive invertible operator $A$. 
Later, Umegaki \cite{umegaki1962} introduced the relative operator entropy 
to discuss the measures of entropy and information. The concept of relative operator entropy for strictly positive operators in noncommutative information theory was initially introduced by Fujii and Kamei in \cite{fujii1989relative,fujii1989uhlmann} as follows:
\begin{align}
\label{roe}
S(A|B)=A^{\frac{1}{2}} [\log(A^{-\frac{1}{2}}BA^{-\frac{1}{2}})]A^{\frac{1}{2}}.
\end{align}
Extending the relative operator entropy, 
Furuta in \cite{furuta2004parametric} defined the generalized relative operator entropy for the strictly positive operators $A, B$ and $\alpha\in \Rdb$ by setting
\begin{align}
\label{groe1}
S_{\alpha}(A|B)=A^{\frac{1}{2}}[(A^{-\frac{1}{2}}BA^{-\frac{1}{2}})^{\alpha} \log(A^{-\frac{1}{2}}BA^{-\frac{1}{2}})]A^{\frac{1}{2}}.	
\end{align}
If $\alpha=0,$ then $S_0(A|B)$ precisely is the relative operator entropy.  
Moreover, Nikoufar \cite{nikoufar_2019} introduced 
the notion of the relative operator $(\alpha, \beta)$-entropy as follows:
\begin{align}
\label{groe2}
S_{\alpha,\beta}(A|B)=
A^{\frac{\beta}{2}}[(A^{-\frac{\beta}{2}}BA^{-\frac{\beta}{2}})^{\alpha}
\log(A^{-\frac{\beta}{2}}BA^{-\frac{\beta}{2}})]A^{\frac{\beta}{2}}
\end{align}
for invertible positive operators $A, B$ and any real numbers $\alpha, \beta.$
In particular, $S_{\alpha, 1}(A|B)=S_{\alpha}(A|B)$ and $S_{0, 1}(A|B)=S(A|B).$ 

In a different direction, Lieb \cite{LIEB1973267} and Lieb and Ruskai \cite{Lieb19731938}
published in 1973 two articles on operator inequalities 
that had a significant impact on quantum mechanics and more recently on quantum information theory. 
Subsequently, much effort were made to explicate and generalize these results. Notably, Effros  \cite{effros2009matrix} presented a simple approach to the operator inequalities of Lieb \cite{LIEB1973267} 
 by utilizing an operator version of perspective function. More specifically, let $f$ and $h$ be real continuous function on a closed interval $I\subset \Rdb$ and $A$ and $B$ two  positive {\em commuting} operators. Effros's perspective function was defined by 
$$({f\Delta h})(A, B)=f\left( \dfrac{A}{h(B)}\right)h(B).$$  
Extending Effros \cite{effros2009matrix}, 
Ebadian and his collaborators \cite{ebadian2011perspectives} introduced a {\em noncommutative} 
perspective function $P_{f\triangle h}(A,B)$ of two variables $A$ and $B$ associated to $f$ and $h$, 
where $A$ is a self-adjoint operator and $B$ is a positive invertible operator on a Hilbert space $H$ with spectra in a closed interval $I$ containing $0$, by setting 
\begin{equation}
\label{persp}
P_{f\triangle h}(A,B)
=h(B)^{\frac{1}{2}}f\left(h(B)^{-\frac{1}{2}}Ah(B)^{-\frac{1}{2}}\right)h(B)^{\frac{1}{2}} 
\end{equation}

Note that in this approach the assumption of the commutativity of $A$ and $B$ is dropped  
and this should have a deep influence on quantum mechanics and quantum information theory.
With such a powerful notion, Ebadian {\sl et al} first proved the necessary and sufficient condition of jointly convexity for noncommutative generalized perspective function {\sl ibid.}, 
then they gave arguably the simplest proof of Lieb concavity theorem and Ando convexity theorem \cite{NIKOUFAR2013531}. More recently, in a series of papers 
\cite{NIKOUFAR2014376, nikoufar2018generalized, nikoufar_2019, Nikoufar2020}, Nikoufar investigated  the convexity property of relative operator $(\alpha, \beta)$-entropy and obtained, later refined its lower and upper bounds by taking advantage of noncommutative perspective function approach.

In the present paper, we study the relative operator entropies in the more general settings of C*-algebras, real C*-algebras and JC-algebras. We show that all the operator inequalities on relative operator entropies still hold  in our settings. Moreover, we refine the lower and upper bounds of the relative operator $(\alpha, \beta)$-entropy established by Nikoufar in \cite{NIKOUFAR2014376} and \cite{Nikoufar2020} 
which improved the bounds for the relative operator entropy obtained 
earlier by Fujii and Kamei \cite{fujii1989relative,fujii1989uhlmann}.

In the rest of this section we give some background and fix the notation. 
The self-adjoint part $B(H)_{sa}$ of all bounded operators 
$B(H)$ on a complex Hilbert space $H$ is a Jordan algebra under the 
{\em non-associative} product $A \circ B := \frac{1}{2}(AB+BA)$.  
A {\bf JC-algebra} is 
a norm-closed  {\bf  Jordan subalgebra} of the Jordan algebra $B(H)_{sa}$; 
it is an ordered space with order inherited from that of $B(H)_{sa}$ via 
the cone $B(H)_{+}$ of positive operators on $H$. 
The notion of JC-algebras arose in the work of Jordan, von Neumann, and Wigner on the axiomatic
foundations of quantum mechanics.  The `observables' in a quantum 
system naturally constitute a JC-algebra and play the central role in the formulation of quantum theory.
For $A, B$ in a unital JC-algebra, we have 
the important identity of Jordan algebras 
\begin{equation}
\label{ija}
ABA=2(A\circ B)\circ A-A^2\circ B
\end{equation} 
If $A$ is strictly positive, then for any $\alpha\in \Rdb,$ $A^{\alpha}$ is contained in JC-algebra generated by $A, 1$,  which is automatically associative. 
More information on JC-algebras can be found in \cite{topping1965jordan}

In this paper, $\Al$ denotes a C*-algebra, or a real C*-algebra, or a JC-algebra, with a unit. 
For two strictly positive elements $A, B$ in $\Al$ and $\alpha, \beta$ in $\Rdb,$ 
the operator $(\alpha, \beta)$-geometric mean is defined as
\begin{equation}
A\#_{(\alpha, \beta)}B=A^{\frac{\beta}{2}}(A^{-\frac{\beta}{2}}BA^{-\frac{\beta}{2}})^{\alpha }A^{\frac{\beta}{2}}
\end{equation}
Note that if $\beta=1,$ then $A\#_{(\alpha, 1)}B$ is the $\alpha$-geometric mean $A\#_{\alpha}B$ in the sense of Ando \cite{ando1979concavity}.  
The crucial observation is that relative operator entropies 
$S(A|B)$, 
$S_{\alpha}(A|B)$, 
$S_{\alpha,\beta}(A|B)$ and 
perspective function   
$P_{f\triangle h}(A,B)$ defined in (\ref{roe}), (\ref{groe1}), (\ref{groe2}), (\ref{persp}) all lie in $\Al$ for  all three cases; in the case of a JC-algebra, this is due to (\ref{ija}) and 
functional calculus \cite{Alfsen2003}.  

\section{Main Results}

For the special case of a $C^*$-algebra $\Al = B(H)$, the following is 
Theorem 1 in \cite{Nikoufar2020}. 

\begin{theorem}
\label{mnp} 
	Let $\Al$ be a $C^*$-algebra, 
	or a real $C^*$-algebra, or a JC-algebra, with unit.	
	Let $r, q, k$ and $h$ be real valued continuous functions on a closed interval $\Idb$ such that $h>0$ 
	and $r(x)\leq q(x)\leq k(x)$. For a positive invertible element $A$ 
	and a self-adjoint element $B$ in $\Al$ such that the spectra of $A$ and $h(A)^{-1/2}Bh(A)^{-1/2}$ are contained in $\Idb$, 
	\begin{equation}
	P_{r\triangle h}(B, A)\leq P_{q\triangle h}(B, A)\leq P_{k\triangle h}(B, A).
	\end{equation}
\end{theorem}

\begin{proof}
	Let $f = q - r$ or  $f = k- q $. Then
	$ f(h(A)^{-\frac{1}{2}}Bh(A)^{-\frac{1}{2}}) \geq 0$ since the three functional calculi are 
	order preserving, cf. \cite{eilers2018c, lireal2003, Alfsen2003}. Therefore, 
	$$ h(A)^{\frac{1}{2}} f(h(A)^{-\frac{1}{2}}Bh(A)^{-\frac{1}{2}}) h(A)^{\frac{1}{2}} \geq 0.$$  
\end{proof}
\begin{remark}
	In general, the spectrum of an element in a real $C^*$-algebra is always symmetric with respect to the real axis. 
	Therefore the three functional calculi for $C^*$-algebra, real $C^*$-algebra, and JC-algebra give rise to two different elements, cf. \cite{eilers2018c, lireal2003, Alfsen2003}. 
	For instance, for a normal element $A$ in a $C^*$-algebra or real $C^*$-algebra, and an appropriate continuous function $f$, $f(A)$ for the $C^*$-algebra is different for $f(A)$ for the 
	real $C^*$-algebra since the spectra of $A$ in these two cases are different in general. 
	
	For self-adjoint element $A$ in a $C^*$-algebra, or real $C^*$-algebra, or JC-algebra, $f(A)$ gives the same element in all three cases.   
	\end{remark}
	
The following is a quick application of Theorem \ref{mnp} 
to three  operator means, namely, weighted harmonic mean, weighted geometric mean, 
and weighted arithmetic mean in our settings, extending \cite{Ando1978topics}.
\begin{proposition}
Let $A, B$ be positive invertible elements in $\Al.$ For any $0\leq \lambda \leq 1,$ 
\begin{align} \label{whga}
\left((1-\lambda)A^{-1}+\lambda B^{-1}\right)^{-1}\leq 
A^{\frac{1}{2}}\left(A^{-\frac{1}{2}}BA^{-\frac{1}{2}}\right)^{\lambda}A^{\frac{1}{2}}
\leq (1-\lambda)A+\lambda B.	
\end{align}	
\end{proposition}
\begin{proof}
One sees that for any $0\leq \lambda \leq 1$
\begin{align}\label{hga1}
{\displaystyle \left((1-\lambda)1+\lambda x^{-1}\right)^{-1}\leq x^{\lambda}\leq  (1-\lambda)1+\lambda x}	
\end{align}
hold for all $x>0.$ Applying Theorem \ref{mnp} to the inequalities (\ref{hga1}) with $h(x)=x$ gives  (\ref{whga}). 
\end{proof}
\begin{remark}
If  ${\displaystyle\lambda=\frac{1}{2}},$ then the inequalities {\rm (\ref{whga})} become the classical ones for harmonic mean, geometric mean and arithmetic mean.  \end{remark}

Suppose that $A, B$ are strictly positive elements in $\Al$, which 
denotes a C*-algebra, or a real C*-algebra, or a JC-algebra, with a unit.  
With notation as in {\rm (\ref{groe2})}, for real numbers $\alpha\geq 0$ and $\beta>0,$ we set 
\begin{align}
\label{I}
{\rm I}&=2A^{\frac{\beta}{2}}
\left[\left(1-2(1+A^{-\frac{\beta}{2}}B A^{-\frac{\beta}{2}})^{-1}\right)(A^{-\frac{\beta}{2}}B A^{-\frac{\beta}{2}})^{\alpha}\right]A^{\frac{\beta}{2}}\\
\label{II}
{\rm II}&=4A\#_{(\alpha, \beta)}B-8A^{\frac{\beta}{2}}[(A^{-\frac{\beta}{2}}B A^{-\frac{\beta}{2}})^{\alpha}\left ( (A^{-\frac{\beta}{2}}B A^{-\frac{\beta}{2}})^{\frac{1}{2}}+1  \right )^{-1}]A^{\frac{\beta}{2}}\\
\label{III}
{\rm III}&=A\#_{(\alpha+\frac{1}{2}, \beta)}B-A\#_{(\alpha-\frac{1}{2}, \beta)}B\\
\label{V}
{\rm V}&=\frac{1}{2}(A\#_{(\alpha+1, \beta)}B-A\#_{(\alpha-1, \beta)}B) 		
\end{align}
Then ${\rm I}, {\rm II}, {\rm III}$, and ${\rm V}$ are in $\Al$ by functional calculi in 
C*-algebra, real C*-algebra, and JC-algebra, \cite{eilers2018c, lireal2003, Alfsen2003}.  

%

\begin{proposition}\label{orlurnp} 
Let $A$ and $B$ be two positive invertible elements in $\Al$. 
Then for $\alpha\geq 0$ and $\beta>0,$ 
\begin{enumerate}
\item[(i)] $A\#_{(\alpha, \beta)}B-A\#_{(\alpha-1, \beta)}B\leq 
{\rm I}\leq A\#_{(\alpha+1, \beta)}B-A\#_{(\alpha, \beta)}B.$

\item[(ii)] $A\#_{(\alpha, \beta)}B-A\#_{(\alpha-1, \beta)}B\leq 
{\rm V}\leq A\#_{(\alpha+1, \beta)}B-A\#_{(\alpha, \beta)}B.$	

\item[(iii)] $A\#_{(0, \beta)}B-A\#_{(-1, \beta)}B\leq A\#_{(\alpha, \beta)}B-A\#_{(\alpha-1, \beta)}B.$
\end{enumerate}
\end{proposition}

\begin{proof}
Here we only give the proof for the case of a unital JC-algebra; 
the proofs for other cases are similar and more straightforward. 

Proof of (i).  
Let 
\begin{align*}
r(x)&= x^{\alpha}-x^{\alpha-1},\\
q(x)&=2\left( 1-\frac{2}{x+1} \right)	x^{\alpha},\\
k(x)&=x^{\alpha+1}-x^{\alpha}.
\end{align*}
Then, 
\begin{equation}
\label{rqk1}
r(x)\leq q(x)\leq k(x)
\end{equation}
for $x>0$ as shown in the proof of  \cite[Proposition 1]{Nikoufar2020}.
%
%

Denoting $h(x)=x^{\beta}$,  we have 
\begin{align*}
P_{r\triangle h}(B, A)&=A\#_{(\alpha, \beta)}B-A\#_{(\alpha-1, \beta)}B \\
P_{q\triangle h}(B, A)&=2A^{\frac{\beta}{2}}
\left[\left(1-2(A^{-\frac{\beta}{2}}B A^{-\frac{\beta}{2}} + 1)^{-1}\right)(A^{-\frac{\beta}{2}}B A^{-\frac{\beta}{2}})^{\alpha}\right]A^{\frac{\beta}{2}}={\rm I}\\
P_{k\triangle h}(B, A)&=A\#_{(\alpha+1, \beta)}B-A\#_{(\alpha, \beta)}B.
\end{align*}
One sees that $P_{r\triangle h}(B, A)$ and $P_{k\triangle h}(B, A)$ are in $\Al.$ Note also that 
$$\left[\left(1-2(A^{-\frac{\beta}{2}}B A^{-\frac{\beta}{2}} + 1)^{-1}\right)(A^{-\frac{\beta}{2}}B A^{-\frac{\beta}{2}})^{\alpha}\right]$$
is an element in the associative JC-subalgebra generated by $A^{-\frac{\beta}{2}}B A^{-\frac{\beta}{2}}$ and $1.$ Hence, 
$$2A^{\frac{\beta}{2}}
\left[\left(1-2(A^{-\frac{\beta}{2}}B A^{-\frac{\beta}{2}} + 1 )^{-1}\right)(A^{-\frac{\beta}{2}}B A^{-\frac{\beta}{2}})^{\alpha}\right]A^{\frac{\beta}{2}}\in\Al.$$
Applying Theorem \ref{mnp} to the inequalities (\ref{rqk1}), we obtain that  
$$A\#_{(\alpha, \beta)}B-A\#_{(\alpha-1, \beta)}B\leq 
{\rm I}\leq A\#_{(\alpha+1, \beta)}B-A\#_{(\alpha, \beta)}B, $$
proving (i). 

For (ii), the inequalities 
\begin{align} \label{rqk2}
 x^{\alpha}-x^{\alpha-1}\leq \frac{1}{2}(x^{\alpha+1}-x^{\alpha-1})\leq  x^{\alpha+1}-x^{\alpha}
\end{align}
hold for all $x>0,$ which can be found in the proof of  \cite[Proposition 1]{Nikoufar2020}. 
Using the perspective functions associated with the three functions in (\ref{rqk2}) with $h$ as in 
the proof of (i) and applying Theorem \ref{mnp}, (ii) follows.

(iii) Note that $1-\frac{1}{x}\leq x^{\alpha}-x^{\alpha-1}$ for all $x>0.$ Applying Theorem \ref{mnp} to this inequality with $h(x)=x^{\beta},$ we know (iii) is true. 
\end{proof}

It is well known that the following Hermite-Hadamard integral inequality, which motivated  
Choquet's theory, plays a significant role in the theory of convex functions, optimization theory and engineering \cite{niculescu2018convex}.  
\begin{lemma}\label{HHI}
	Let $f$ be a convex function on the interval $[a, b].$ Then,
	\begin{align*}
	f\left(\frac{a+b}{2}\right)\leq \frac{1}{b-a}\int_a^bf(t)dt\leq \frac{f(a)+f(b)}{2}.	
	\end{align*}
\end{lemma}
Nowadays much attention is paid to generalize and extend this classical fundamental result for convex function. For example, El Farissi gives the following improved estimation in \cite{el2010simple}.
\begin{lemma}\label{RHHI}
	Assume that $f: [a, b]\to \Rdb$ is a convex function. Then we have
	\begin{align*}
	f\left(\frac{a+b}{2}\right)\leq \sup_{\lambda\in[0, 1]}l(\lambda)\leq \frac{1}{b-a}\int_a^bf(t)dt\leq \inf_{\lambda\in[0, 1]}L(\lambda) \leq \frac{f(a)+f(b)}{2},	\end{align*}
	where 
	\begin{align*}
	l(\lambda)= \lambda  f\left( \frac{ \lambda b+(2-\lambda)a}{2}\right)+(1-\lambda) f\left(\frac{(1+\lambda) b+(1-\lambda)a}{2}\right)	
	\end{align*}
	and
	\begin{align*}
	\displaystyle 
	L(\lambda)= \dfrac{1}{2} \left[ f( \lambda b+(1-\lambda)a)+\lambda f(a)+(1-\lambda)f(b)\right]. 
	\end{align*}
\end{lemma}

Theorem 2 in \cite{Nikoufar2020} states that for $\alpha\geq 0$, $\beta>0,$ and 
positive invertible operators $A$ and $B$ in $B(H)$ with $A^{\beta}\leq B,$ 
\begin{align}\label{N20T2}
{\rm I}\leq S_{\alpha, \beta}(A|B) \leq {\rm V}.	
\end{align}

The following result refines (\ref{N20T2}) and extends it to our broader settings.

\begin{theorem}\label{ulbroe}
Let $\alpha\geq 0,$ $\beta>0,$
and let $\Al,$ be a $C^*$-algebra or a real $C^*$-algebra or a JC-algebra, with a unit. 	
If $A$ and $B$ are two positive invertible elements in  $\Al$ satisfying $A^{\beta}\leq B,$  then
\begin{align}
\label{main1}
{\rm I}\leq {\rm II}\leq  S_{\alpha, \beta}(A|B)\leq {\rm III}\leq {\rm V}.
\end{align}
\end{theorem}

\begin{proof}
For $\alpha\geq 0$ and $x\geq 1,$ 
define $f(t)=\dfrac{x^{\alpha}}{t}-1$, $1 \leq t \leq x$.
 Then $l(\lambda)$ and $L(\lambda)$ defined in Lemma \ref{RHHI} are
\begin{align*}
l(\lambda)&= x^{\alpha} \dfrac{2\lambda}{\lambda(x-1)+2}+	x^{\alpha} \dfrac{2-2\lambda}{\lambda(x-1)+(x+1)}-1;\\
L(\lambda)&=\dfrac{1}{2}\left( \dfrac{x^{\alpha}}{\lambda(x-1)+1}+\lambda x^{\alpha} +
(1-\lambda)x^{\alpha -1} \right)-1.
\end{align*}
Applying refined Hermite-Hadamard inequality (see Lemma \ref{RHHI}) for $f$ on the interval $[1,x],$ we get
\begin{align*}
f\left(\dfrac{x+1}{2}\right)\leq \sup_{\lambda\in[0,1]}l(\lambda)\leq \dfrac{1}{x-1}\int_1^x \left( \dfrac{x^{\alpha}}{t}-1 \right) dt\leq \inf_{\lambda\in[0,1]}L(\lambda)\leq \dfrac{f(x)+f(1)}{2}.
\end{align*}

By extreme value theorem, it is easy to show that 
\begin{align*}
\sup_{\lambda\in[0,1]}l(\lambda)&=l\left(\frac{1}{\sqrt{x}+1}\right)=\dfrac{4x^{\alpha}}{(\sqrt{x}+1)^2}-1,\\
\inf_{\lambda\in[0,1]}L(\lambda)&=L\left( \frac{1}{\sqrt{x}+1} \right)=\dfrac{x^{\alpha}}{\sqrt{x}}-1.\\
\end{align*}
And, a simple calculation shows that 
\begin{align*}
f\left( \dfrac{x+1}{2} \right) &=\dfrac{2x^{\alpha}}{x+1}-1,\\
\int_1^x \left( \dfrac{x^{\alpha}}{t}-1 \right)dt&=x^{\alpha}\ln x-(x-1),\\	
\dfrac{f(x)+f(1)}{2}&=\dfrac{1}{2} \left( x^{\alpha}+x^{\alpha-1} \right)-1.	
\end{align*}
Therefore, 
\begin{align*} 
\dfrac{2x^{\alpha}}{x+1}-1&\leq \dfrac{4x^{\alpha}}{(\sqrt{x}+1)^2}-1
\leq \dfrac{1}{x-1} \left[ x^{\alpha}\ln x-(x-1) \right]\\
&\leq \dfrac{x^{\alpha}}{\sqrt{x}}-1
\leq \frac{1}{2}(x^{\alpha}+x^{\alpha-1})-1. 
\end{align*}
Let 
\begin{align*}
r(x)&= 2 \left( 1-\dfrac{2}{x+1} \right) x^{\alpha},\\
s(x)&=4x^{\alpha}-\dfrac{8x^{\alpha}}{\sqrt{x}+1},\\
q(x)&=x^{\alpha}\ln(x),\\
j(x)&=\dfrac{x^{\alpha}(x-1)}{\sqrt{x}},\\
k(x)&=\dfrac{1}{2} \left( x^{\alpha+1}-x^{\alpha-1} \right).
\end{align*}
Then 
\begin{equation}
\label{rsqjk1}
r(x)\leq s(x)\leq q(x)\leq j(x)\leq k(x).
\end{equation}

Denoting $h(t)=t^{\beta}$,  we have 
\begin{align*}
P_{r\triangle h}(B, A)&= 2A^{\frac{\beta}{2}}\big(1-2(1+A^{-\frac{\beta}{2}}B A^{-\frac{\beta}{2}}\big)^{-1})(A^{-\frac{\beta}{2}}B A^{-\frac{\beta}{2}})^{\alpha}A^{\frac{\beta}{2}}={\rm I}\\
P_{s\triangle h}(B, A)&=4A\#_{(\alpha, \beta)}B-8A^{\frac{\beta}{2}}[(A^{-\frac{\beta}{2}}B A^{-\frac{\beta}{2}})^{\alpha}\big( (A^{-\frac{\beta}{2}}B A^{-\frac{\beta}{2}})^{\frac{1}{2}}+1 \big)^{-1}]A^{\frac{\beta}{2}}={\rm II}	\\
P_{q\triangle h}(B, A)&=S_{\alpha,\beta}(A|B)\\
P_{j\triangle h}(B, A)&=A_{(\alpha+\frac{1}{2}, \beta)}B-A_{(\alpha-\frac{1}{2}, \beta)}B={\rm III}\\
P_{k\triangle h}(B, A)&=\dfrac{1}{2}(A\#_{(\alpha+1, \beta)}B-A\#_{(\alpha-1, \beta)}B)={\rm V}.
\end{align*}
Applying Theorem \ref{mnp} to (\ref{rsqjk1}), we derive the inequality
\begin{align*}
 {\rm I}\leq {\rm II}\leq  S_{\alpha, \beta}(A|B)\leq {\rm III}\leq {\rm V}.
\end{align*} 
\end{proof}
Coupling Theorem \ref{ulbroe} with  Proposition \ref{orlurnp}, we obtain 
\begin{corollary}\label{crjropi}
If $A$ and $B$ are two positive invertible elements in $\Al$, with $A^{\beta}\leq B,$ and $\alpha\geq 0,$ $\beta>0,$ then 
\begin{align*}
A\#_{(\alpha, \beta)}B-A\#_{(\alpha-1, \beta)}B
&\leq{\rm I}
\leq {\rm II}\\
&\leq  S_{\alpha, \beta}(A|B)\\
&\leq {\rm III}\leq {\rm V}\\
&\leq A\#_{(\alpha+1, \beta)}B-A\#_{(\alpha, \beta)}B.
\end{align*}
\end{corollary}

Corollary 5 in \cite{fujii1989uhlmann} states that for
positive invertible operators $A$ and $B$ in $B(H)$,  
\begin{align}\label{oropi}
A-AB^{-1}A\leq S(A\vert B)\leq B-A.
\end{align}
Under the assumption $A\leq B$,  
\cite{Nikoufar2020} provides the following sharper result for the relative operator entropy: 
\begin{align}\label{oropi1}
A-AB^{-1}A &\leq 2(A-2A(A+B)^{-1}A) \leq S(A\vert B)  \nonumber \\
&\leq \frac{1}{2} \left(B- AB^{-1}A \right) \leq B-A.	
\end{align}
The following result refines (\ref{oropi}), (\ref{oropi1}), and extends them to our settings.
\begin{corollary}\label{joropi}
If $A$ and $B$ are two positive invertible elements in $\Al$, with $A\leq B,$ then 
\begin{align*}
A-AB^{-1}A
&\leq 2(A-2A(A+B)^{-1}A)\\
&\leq 4A-8A(A\#_{\frac{1}{2}}B+A)^{-1}A\\
&\leq S(A|B)\\
&\leq A\#_{\frac{1}{2}}B-A\#_{-\frac{1}{2}}B\\
&\leq \frac{1}{2}(B-AB^{-1}A)\\
&\leq B-A.
\end{align*}
\end{corollary}

\begin{theorem}\label{ulbroe1}
If $A$ and $B$ are two positive invertible elements in $\Al$ with $A^{\beta}\geq B,$ and $\alpha\geq 0,$ $\beta>0,$ then 
\begin{align}
\label{main2}
{\rm V}\leq {\rm III}\leq  S_{\alpha, \beta}(A|B)\leq {\rm II}\leq {\rm I}.
\end{align}
\end{theorem}

\begin{proof}
	For $\alpha\geq 0$ and $0 < x \leq 1,$ 
	define $f(t)=\dfrac{x^{\alpha}}{t}-1,$ $t \in [x, 1]$.  Then $l(\lambda)$ and $L(\lambda)$ defined in Lemma \ref{RHHI} are
\begin{align*}
l(\lambda)&= x^{\alpha} \dfrac{2\lambda}{\lambda(1 - x)+2x} +	
x^{\alpha} \dfrac{2(1-\lambda)}{\lambda(1 - x)+(x+1)}-1;\\
L(\lambda)&=\dfrac{1}{2} \left( \dfrac{x^{\alpha}}{\lambda(1-x)+x}+\lambda x^{\alpha-1}+(1-\lambda)x^{\alpha} \right)-1.
\end{align*}
By applying refined Hermite-Hadamard inequality (see Lemma \ref{RHHI}) for $f$ on the interval $[x,1],$ we get
\begin{align*}
f \left( \dfrac{x+1}{2} \right) \leq \sup_{\lambda\in[0,1]}l(\lambda)\leq \dfrac{1}{1-x}\int_x^1 \left( \dfrac{x^{\alpha}}{t}-1 \right) dt \leq \inf_{\lambda\in[0,1]}L(\lambda)\leq \dfrac{f(x)+f(1)}{2}.
\end{align*}

By extreme value theorem, we know that 
\begin{align*}
\sup_{\lambda\in[0,1]}l(\lambda)&=l \left( \frac{\sqrt{x}}{\sqrt{x}+1} \right) = \dfrac{4x^{\alpha}}{(\sqrt{x}+1)^2}-1,\\
\inf_{\lambda\in[0,1]}L(\lambda)&=L \left( \frac{\sqrt{x}}{\sqrt{x}+1} \right)=\dfrac{x^{\alpha}}{\sqrt{x}}-1.\\
\end{align*}
Thus,
\begin{align*}
\dfrac{2x^{\alpha}}{x+1}-1&\leq \dfrac{4x^{\alpha}}{(\sqrt{x}+1)^2}-1
\leq \dfrac{1}{x-1} \left[ x^{\alpha}\ln x-(x-1) \right]\\
 &\leq \dfrac{x^{\alpha}}{\sqrt{x}}-1
 \leq \dfrac{1}{2} \left( x^{\alpha}+x^{\alpha-1} \right)-1.
\end{align*}
Let 
\begin{align*}
r(x)&= 2 \left( 1-\dfrac{2}{x+1} \right) x^{\alpha},\\
s(x)&=4x^{\alpha}-\dfrac{8x^{\alpha}}{\sqrt{x}+1},\\
q(x)&=x^{\alpha}\ln(x),\\
j(x)&=\dfrac{x^{\alpha}(x-1)}{\sqrt{x}},\\
k(x)&=\dfrac{1}{2} \left( x^{\alpha+1}-x^{\alpha-1} \right).
\end{align*}
Then 
\begin{equation}
\label{rsqjk2}
k(x)\leq j(x)\leq q(x)\leq s(x)\leq r(x).
\end{equation}
Applying Theorem \ref{mnp} to (\ref{rsqjk2}), the desired inequalities follow.
\end{proof}
Similarly, combing Theorem \ref{ulbroe1} with Proposition \ref{orlurnp}, we have  
\begin{corollary}\label{joropi2}
If $A$ and $B$ are two positive invertible elements in $\Al$, with $A^{\beta}\geq B,$ and $\alpha\geq 0,$ $\beta>0,$ then 
\begin{align*}
A\#_{(\alpha, \beta)}B-A\#_{(\alpha-1, \beta)}B
&\leq  {\rm V}\leq {\rm III}\\
&\leq S_{\alpha, \beta}(A|B)\\
&\leq {\rm II}\leq {\rm I}\\
&\leq A\#_{(\alpha +1 , \beta)}B-A\#_{(\alpha, \beta)}B.
\end{align*}

\end{corollary}
The following improves Corollary 5 in \cite{fujii1989uhlmann} and extends it to our more general settings.
\begin{corollary}\label{joropi3}
If $A$ and $B$ are two positive invertible elements in $\Al$, with $A\geq B,$ then
\begin{align*}
A-AB^{-1}A&\leq \frac{1}{2}(B-AB^{-1}A)\\
&\leq A\#_{\frac{1}{2}}B-A\#_{-\frac{1}{2}}B\\
&\leq S(A|B)\\
&\leq 4A-8A(A\#_{\frac{1}{2}}B+A)^{-1} A\\
&\leq 2(A-2A(A+B)^{-1}A)\\
&\leq B-A.
\end{align*}
 
\end{corollary}
\section{Further Refinements}
Suppose that $A, B$ are two positive invertible elements. For any real number $\delta>0,$ $\alpha\geq 0$ and $\beta>0,$ we denote
\begin{align*}
{\rm I'}&=(\ln\delta+2)A\#_{(\alpha,\beta)} B-2\delta A^{\frac{\beta}{2}}[(A^{-\frac{\beta}{2}}BA^{-\frac{\beta}{2}}+\delta)^{-1}(A^{-\frac{\beta}{2}}BA^{-\frac{\beta}{2}})^{\alpha}]A^{\frac{\beta}{2}}\\
{\rm II'}&=(\ln\delta+4)A\#_{(\alpha,\beta)} B-8\sqrt{\delta}A^{\frac{\beta}{2}}[\big((A^{-\frac{\beta}{2}}BA^{-\frac{\beta}{2}})^{\frac{1}{2}}+\sqrt{\delta}\big)^{-1}(A^{-\frac{\beta}{2}}BA^{-\frac{\beta}{2}})^{\alpha}]A^{\frac{\beta}{2}}\\
{\rm III'}&=(\frac{1}{\sqrt{\delta}}A\#_{(\alpha+\frac{1}{2},\beta)}B-\sqrt{\delta} A\#_{(\alpha-\frac{1}{2},\beta)}B)+\ln\delta A\#_{(\alpha,\beta)}B\\
{\rm V'}&=\frac{1}{2}(\frac{1}{\delta}A\#_{(\alpha+1,\beta)}B-\delta A\#_{(\alpha-1,\beta)}B)+\ln\delta A\#_{(\alpha,\beta)}B.
\end{align*}

\begin{theorem}
	\label{ulbroe2}
Let $A$ and $B$ be two positive invertible elements in $\Al$ which is a $C^*$-algebra or real $C^*$ or JC-algebra $\Al,$ with unit, satisfying $ \delta\geq 1, \delta A^{\beta}\leq B,$ and $\alpha\geq0$ and $\beta\geq 0.$ Then
\begin{equation}
{\rm I'}\leq {\rm II'}\leq S_{\alpha, \beta}(A|B)\leq {\rm III'}\leq {\rm V'}.
\end{equation}
For  $\delta\leq 1$ and $B\leq \delta A^{\beta}$, the reversed inequalities hold.
\end{theorem}
\begin{proof}
Let $x\geq \delta.$ Then $x/{\delta}\geq 1.$ By substituting $x$ with $x/{\delta},$ the following inequalities hold:
\begin{enumerate}
\item[(a)]$\left( \ln \delta+2(1-\dfrac{2\delta}{x+\delta}) \right)x^{\alpha}	 
\leq \left( \ln \delta+4-\dfrac{8\sqrt{\delta}}{\sqrt{x}+\sqrt{\delta}} \right) x^{\alpha}	\leq x^{\alpha}	\ln\alpha.$
\item[(b)] $x^{\alpha}\ln\alpha \leq x^{\alpha+\frac{1}{2}}\dfrac{1}{\sqrt{\delta}}-x^{\alpha-\frac{1}{2}}\sqrt{\delta}+x^{\alpha}\ln\delta\leq\dfrac{x^{\alpha+1}}{2\delta}-\dfrac{x^{\alpha -1}}{2}\delta+x^{\alpha}\ln \delta.$
\end{enumerate}

Let 
\begin{align*}
r(x)&= \left[ \ln \delta+2 \left( 1-\dfrac{2\delta}{x+\delta} \right) \right] x^{\alpha}	,\\
s(x)&=\left[ \ln \delta+4-\dfrac{8\sqrt{\delta}}{\sqrt{x}+\sqrt{\delta}} \right] x^{\alpha},\\
q(x)&=x^{\alpha}\ln(x),\\
j(x)&=x^{\alpha+\frac{1}{2}}\dfrac{1}{\sqrt{\delta}}-x^{\alpha-\frac{1}{2}}\sqrt{\delta}+x^{\alpha}\ln\delta,\\
k(x)&=\dfrac{x^{\alpha+1}}{2\delta}-\dfrac{x^{\alpha -1}}{2}\delta+x^{\alpha}\ln \delta,\\
h(t)&=t^{\beta}.
\end{align*}
Then
\begin{align*}
P_{r\triangle h}(B,A)&=(\ln\delta+2)A\#_{(\alpha,\beta)} B-2\delta A^{\frac{\beta}{2}}[(A^{-\frac{\beta}{2}}BA^{-\frac{\beta}{2}}+\delta)^{-1}(A^{-\frac{\beta}{2}}BA^{-\frac{\beta}{2}})^{\alpha}]A^{\frac{\beta}{2}}={\rm I'}\\
P_{s\triangle h}(B,A)&=(\ln\delta+4)A\#_{(\alpha,\beta)} B-8\sqrt{\delta}A^{\frac{\beta}{2}}[\big((A^{-\frac{\beta}{2}}BA^{-\frac{\beta}{2}})^{\frac{1}{2}}+\sqrt{\delta}\big)^{-1}(A^{-\frac{\beta}{2}}BA^{-\frac{\beta}{2}})^{\alpha}]A^{\frac{\beta}{2}}={\rm II'}\\
P_{s\triangle h}(B,A)&=S_{\alpha,\beta}(A|B)\\
P_{j\triangle h}(B,A)&= \left( \frac{1}{\sqrt{\delta}}A\#_{(\alpha+\frac{1}{2},\beta)}B-\sqrt{\delta} A\#_{(\alpha-\frac{1}{2},\beta)}B \right)+\ln\delta A\#_{(\alpha,\beta)}B={\rm III'}\\
P_{k\triangle h}(B,A)&=\frac{1}{2} \left(\frac{1}{\delta}A\#_{(\alpha+1,\beta)}B-
\delta A\#_{(\alpha-1,\beta)}B \right)+\ln\delta A\#_{(\alpha,\beta)}B={\rm V'}.
\end{align*}
The results follow immediately by applying Theorem \ref{mnp}. 
\end{proof}

Theorem \ref{ulbroe2} above refines Theorem 3 in \cite{Nikoufar2020} and extends it to our broader settings, 
where it is shown 
$${\rm I'} \leq S_{\alpha, \beta}(A|B)\leq {\rm V'}.$$

\begin{proposition}
	\label{ulbroe3}
For  $\alpha\geq 0$, $\beta>0$, $\delta\geq 1$, and positive invertible elements $A$ and $B$ in $\Al$ 
with $\delta A^{\beta}\leq B,$ we have
\begin{align*}
{\rm II}\leq {\rm II'}\quad \mbox{and}\quad {\rm III'}\leq {\rm III},
\end{align*}	
where {\rm II} and {\rm III} are defined in {\rm (\ref{II})} and {\rm (\ref{III})} respectively.  
For $\delta\leq 1$ and $B\leq \delta A^{\beta}$ the reversed inequalities hold.
\end{proposition}
\begin{proof}
(i) Proof of ${\rm II}\leq {\rm II'}$. Denote $\Gamma(t)= \left( \ln t+4-\dfrac{8\sqrt{t}}{\sqrt{x}+\sqrt{t}} \right) x^{\alpha}$ and $x\geq \delta\geq t\geq 1.$ Then $$\Gamma'(t)=\dfrac{(x^{\frac{1}{2}}-t^{\frac{1}{2}})^2}{(x^{\frac{1}{2}}+t^{\frac{1}{2}})^2t}\geq 0.$$	
Suppose $x$ is fixed, then 
$\Gamma(1)\leq \Gamma(\delta),$ that is,
\begin{align}\label{sharp1}
\left( 4-\dfrac{8}{\sqrt{x}+1} \right) x^{\alpha}
\leq \left( \ln t+4-\dfrac{8\sqrt{t}}{\sqrt{x}+\sqrt{t}} \right) x^{\alpha}	
\end{align}

Let $h(t)=t^{\beta}.$ By applying Theorem \ref{mnp} to the inequality (\ref{sharp1}), 
the result follows directly. 

(ii) Proof of $ {\rm III'}\leq {\rm III}$. Let $\Omega(t)=x^{\alpha+\frac{1}{2}}\dfrac{1}{\sqrt{t}}-x^{\alpha-\frac{1}{2}}\sqrt{t}+x^{\alpha}\ln t$ with $x\geq \delta\geq t\geq 1.$
Then $$\Omega'(t)=-\dfrac{x^{\alpha}}{2t^{\frac{3}{2}}} 
\left( x^{\frac{1}{4}}-x^{-\frac{1}{4}}t^{\frac{1}{2}} \right)^2\leq 0.$$
Suppose $x$ is fixed. Then $\Omega(\delta)\leq \Omega(1)$, which is equivalent to
\begin{align}\label{sharp2}
x^{\alpha+\frac{1}{2}}\frac{1}{\sqrt{\delta}}-x^{\alpha -\frac{1}{2}}\sqrt{\delta}+x^{\alpha}\ln \delta\leq x^{\alpha+\frac{1}{2}}-x^{\alpha-\frac{1}{2}}.\end{align}
Let $h(t)=t^{\beta}.$ By applying Theorem \ref{mnp} to the inequality ($\ref{sharp2}$), 
the result follows immediately. 
\end{proof}

Combining Theorem \ref{ulbroe2} and Proposition \ref{ulbroe3}, we have

\begin{corollary}
Let $A$ and $B$ be two positive invertible elements in $\Al$ which is a $C^*$-algebra or real $C^*$-algebra or JC-algebra, with a unit, let $\alpha \geq 0, \beta>0$, and let {\rm I} and {\rm V} be as in 
{\rm (\ref{I})} and {\rm (\ref{V})} respectively.
\begin{itemize}
\item[(i)] If  $\delta\geq 1$ and $\delta A^{\beta}\leq B$, then
\begin{align*}
{\rm I}\leq {\rm II}\leq {\rm II'}\leq S_{\alpha, \beta}(A|B)\leq {\rm III'}\leq {\rm V'}\leq {\rm V}.	
\end{align*}
\item[(ii)] If $\delta\leq 1$ and $\delta A^{\beta}\geq B$, then
\begin{align*}
 {\rm V}\leq {\rm V'}\leq {\rm III'}\leq S_{\alpha, \beta}(A|B)\leq {\rm II'}\leq {\rm II}\leq {\rm I}.	
\end{align*}	
	
\end{itemize}
\end{corollary}
The following corollary provides a sharper refinement in our settings for the lower and upper bounds of the relative operator entropy established by Nikoufar in \cite{NIKOUFAR2014376, Nikoufar2020},  
which improved the bounds obtained 
earlier by Fujii and Kamei \cite{fujii1989relative,fujii1989uhlmann}.
\begin{corollary}
Let $A$ and $B$ be two positive invertible elements in $\Al$ which is a $C^*$-algebra or real $C^*$-algebra or JC-algebra, with a unit.
\begin{itemize}
\item[(i)] If $\delta\geq 1$ and $\delta A\leq B$, then 
\begin{align*}
A-AB^{-1}A
&\leq 2[A-2A(A+B)^{-1}A]\\
&\leq (\ln\delta)A+2[A-2\delta A(\delta A+B)^{-1}A]\\
&\leq (\ln\delta+4)A-8\sqrt{\delta}A(A\#_{\frac{1}{2}}B+\sqrt{\delta}A)^{-1}A \\
&\leq S(A|B)\\
&\leq \left( \frac{1}{\sqrt{\delta}}A\#_{\frac{1}{2}}B-\sqrt{\delta} A\#_{-\frac{1}{2}}B \right)+(\ln\delta) A\\
&\leq \frac{1}{2} \left( \frac{1}{\delta}B-\delta AB^{-1}A \right)+(\ln\delta) A\\
&\leq \frac{1}{2}(B-AB^{-1}A)\\
&\leq B-A.
\end{align*}
\item[(ii)] If $\delta\leq 1$ and $\delta A\geq B$, then 
\begin{align*}
A-AB^{-1}A
&\leq \frac{1}{2}(B-AB^{-1}A)\\
&\leq \frac{1}{2} \left( \frac{1}{\delta}B-\delta AB^{-1}A \right)+(\ln\delta) A\\
&\leq \left( \frac{1}{\sqrt{\delta}}A\#_{\frac{1}{2}}B-\sqrt{\delta} A\#_{-\frac{1}{2}}B \right)+(\ln\delta) A\\
&\leq  S(A|B)\\
&\leq (\ln\delta+4)A-8\sqrt{\delta}A(A\#_{\frac{1}{2}}B+\sqrt{\delta}A)^{-1}A \\
&\leq (\ln\delta)A+2[A-2\delta A(\delta A+B)^{-1}A]\\
&\leq 2[A-2A(A+B)^{-1}A]\\
&\leq B-A.
\end{align*}		
\end{itemize}
\end{corollary}

\bibliographystyle{abbrv}
\bibliography{S.Wang_and_Z.Wang_ROE.bib}

\end{document}